\documentclass[11pt]{amsart}
\usepackage[T1]{fontenc}
\usepackage[utf8]{inputenc} 
\usepackage[
  backend=biber,
  style=numeric,
  sorting=none,
  giveninits=false,
  doi=true,
  isbn=true,
  url=true,
  eprint=true
]{biblatex}

\usepackage{lmodern}
\usepackage{microtype}

\usepackage{amsmath, amssymb, amsfonts, amsthm}
\usepackage{mathtools}
\usepackage{enumitem}
\usepackage{hyperref}
\usepackage{cleveref}
\usepackage{tikz-cd}
\usepackage{geometry}
\usepackage{thmtools}
\usepackage{array}
\usepackage{caption}
\newcolumntype{C}[1]{>{\centering\arraybackslash}m{#1}}
\usepackage{float}
\usepackage{caption}

\hypersetup{
  colorlinks=true,
  linkcolor=blue,
  citecolor=blue,
  urlcolor=blue,
  pdftitle={Perverse Schober Models for Conifold Degenerations and Duality-Preserving Invariants},
  pdfauthor={Abdul Rahman}
}
\numberwithin{equation}{section}
\theoremstyle{plain}
\newtheorem{theorem}{Theorem}[section]
\newtheorem{proposition}[theorem]{Proposition}

\newtheorem{corollary}[theorem]{Corollary}
\newtheorem{conjecture}[theorem]{Conjecture}
\theoremstyle{definition}

\theoremstyle{remark}
\newtheorem{remark}{Remark}[section]

\newcommand{\Q}{\mathbb{Q}}
\newcommand{\Z}{\mathbb{Z}}
\newcommand{\C}{\mathbb{C}}

\newcommand{\Perv}{\mathrm{Perv}}

\newcommand{\Ext}{\mathrm{Ext}}

\newcommand{\var}{\mathrm{var}}
\newcommand{\rat}{\mathrm{rat}}

\newcommand{\HH}{\mathrm{HH}}
\newcommand{\Gr}{\mathrm{Gr}}

\DeclareMathOperator{\Cone}{Cone}

\DeclareMathOperator{\id}{id}

\title[Perverse Extensions and Limiting Mixed Hodge Structures]{Perverse Extensions and Limiting Mixed Hodge Structures for Conifold Degenerations}

\author{Abdul Rahman}
\thanks{Email: arahman@alum.howard.edu}

\begin{document}

\begin{abstract}
Let \(\pi:X\to\Delta\) be a one-parameter degeneration whose central fiber \(X_0\) has a single ordinary double point. The nearby- and vanishing-cycle formalism determines a canonical perverse sheaf on \(X_0\), obtained from the variation morphism and fitting into an extension of the intersection complex by a point-supported rank-one contribution. We study this object from the perspective of limiting mixed Hodge theory and Saito's theory of mixed Hodge modules. In the ordinary double point case, we show that the corrected perverse object is the unique minimal Verdier self-dual perverse extension of the shifted constant sheaf across the node, and that its rank-one singular contribution and the corresponding rank-one vanishing contribution in the limiting mixed Hodge structure arise from the same nearby-cycle formalism. We also formulate the analogous structural statements for multi-node degenerations and for more general stratified singular loci. Finally, we explain how Saito's divisor-gluing formalism provides the natural framework for a fuller mixed-Hodge-module refinement of these constructions.
\end{abstract}

\maketitle

\tableofcontents
\section{Introduction}
Degenerations of complex varieties produce two closely related kinds of data. On the one hand,  the nearby and vanishing cycle functors encode the topological change in the family and provide a  natural sheaf-theoretic framework for studying the singular fiber. On the other hand, when the 
family carries Hodge-theoretic structure, the same degeneration gives rise to a limiting mixed Hodge  structure governed by monodromy. In the case of a one-parameter degeneration with an ordinary double point, these two viewpoints meet in a particularly transparent way through Picard--Lefschetz 
theory and the rank-one vanishing-cycle contribution. In earlier work \cite{RahmanSchoberPaper}, we studied a one-parameter conifold degeneration
\[
\pi:X \to \Delta
\]
whose central fiber $X_0$ has a single ordinary double point, and we associated to it the perverse 
sheaf
\[
\mathcal P := \Cone\!\bigl(\var_F:\phi_\pi(F)\to\psi_\pi(F)\bigr)[-1],
\qquad F := \mathbf{Q}_X[3].
\]
That paper showed that $\mathcal P$ is a canonical perverse object on $X_0$, determined functorially by the  nearby/vanishing-cycle triangle, and that in the ordinary double point case it is an  extension of the intersection complex by a rank-one skyscraper contribution supported at the node \cite{KapranovSchechtman,MacPhersonVilonen1986,GMV1996}. In particular, $\mathcal P$ restricts to the  shifted constant sheaf on the smooth locus and records the rank-one vanishing contribution detected by the Milnor fiber.

The purpose of the present paper is to study the Hodge-theoretic content of the canonical perverse extension through the nearby-cycle formalism in Saito's theory of mixed Hodge modules. Rather than identifying extension classes directly  across different categories, we study the mixed-Hodge-module data carried by nearby and vanishing cycles and use the realization functor
\[
\mathrm{rat}: MHM(X_0)\to \Perv(X_0)
\]
to relate the Hodge-theoretic and perverse-sheaf-theoretic pictures. The key structural input is 
Saito's divisor-case gluing formalism for mixed Hodge modules, which describes objects along a 
principal divisor in terms of data on the complement, data on the divisor, and morphisms controlled 
by the nilpotent monodromy operator \cite{SaitoMHM}. We also prove $\mathcal P$ is Verdier self-dual in Section \ref{sec:uniq-verdier-self-dual}. 

\subsection{Relation to earlier work and Hodge-theoretic framework}

The use of perverse sheaves in the study of singular Calabi--Yau spaces and conifold transitions goes back to earlier attempts to construct cohomological models that retain duality-theoretic features in the presence of singularities. Foundational work of Beilinson--Bernstein--Deligne, together with the linear-algebraic descriptions of perverse
sheaves developed by MacPherson--Vilonen and Gelfand--MacPherson--Vilonen, provides the basic formalism for treating a space with a single singular stratum in terms of gluing and extension data \cite{BBD,MacPhersonVilonen1986,GMV1996}. In the conifold setting, such ideas
have already appeared in the construction of perverse-sheaf models for corrected cohomology theories and in the study of singular Calabi--Yau threefolds arising in string theory \cite{RahmanATMP,BanaglIS,BanaglBudurMaxim}.

In \cite{RahmanSchoberPaper}, we  approached a one-parameter conifold degeneration from the nearby/vanishing-cycle side and isolated the canonical perverse sheaf
\[
\mathcal P :=\Cone(\var_F)[-1]
\]
as a natural perverse object on the singular fiber in the ordinary double point case. In that setting, the main point was the existence of a canonical perverse extension determined by the variation morphism between vanishing and nearby cycles, together with its relation to the rank-one local contribution coming from Picard--Lefschetz monodromy. The emphasis in \cite{RahmanSchoberPaper} was primarily sheaf-theoretic: the object $\mathcal P$ was constructed and analyzed inside the category of rational perverse sheaves, and the mixed-Hodge-theoretic refinement was left as a further direction.

The present paper takes up that Hodge-theoretic question using Saito's theory of mixed Hodge modules \cite{SaitoMHM,SaitoDuality}. For a complex algebraic variety $X$, Saito constructs an abelian category $MHM(X)$ together with an exact and faithful realization functor
\[
\rat \colon MHM(X)\to \Perv(X),
\]
to the category of rational perverse sheaves \cite{SaitoMHM}. In particular, the nearby-cycle and vanishing-cycle functors admit lifts to the mixed-Hodge-module setting and are compatible, under $\rat$, with the corresponding functors on perverse sheaves \cite{SaitoMHM}. This is the basic formal mechanism that allows one to compare the Hodge-theoretic degeneration data of a family with the canonical perverse sheaf attached to its singular fiber.

A second ingredient from Saito, especially relevant here, is the divisor-case gluing formalism. If $Y:=g^{-1}(0)$ is a principal divisor in a variety $X$ and $U:=X\setminus Y$, then mixed Hodge modules along $Y$ may be described in terms of data on $U$, data on $Y$, and morphisms relating nearby-cycle information, with compatibility governed by the nilpotent
monodromy operator \cite{SaitoMHM}. Since the central fiber
\[
X_0:=\pi^{-1}(0)
\]
of a one-parameter degeneration is a principal divisor, this formalism provides the natural structural setting to reinterpret the canonical perverse extension arising from nearby cycles. For the purposes of the present paper, we use only the general consequences of Saito's
theory that are needed for this reinterpretation: the existence of the categories $MHM(X)$, the exact faithful functor $\rat$, the mixed-Hodge-module nearby- and vanishing-cycle functors, and the divisor-case gluing formalism \cite{SaitoMHM}.

An important precedent for this point of view is the work of Banagl--Budur--Maxim \cite{BanaglBudurMaxim}. In their setting, for a projective hypersurface with an isolated singularity, they construct a perverse sheaf whose hypercohomology computes the intersection-space cohomology and show that this perverse sheaf underlies a mixed Hodge module, so that its hypercohomology inherits canonical mixed Hodge structures \cite{BanaglBudurMaxim}. Under an additional semisimplicity hypothesis on the local monodromy for the eigenvalue $1$, they further obtain a splitting of the nearby-cycle perverse sheaf in which their intersection-space complex appears as a direct summand \cite{BanaglBudurMaxim}. Although the object studied in \cite{BanaglBudurMaxim} is different from the canonical perverse extension considered here, that paper provides a useful model for
how a perverse sheaf built from degeneration data can be placed in a mixed-Hodge-module framework without identifying the two categories.

The relation between \cite{BanaglBudurMaxim} and the present paper is therefore one of method rather than of direct equivalence of objects. Their intersection-space complex is designed to recover intersection-space cohomology for isolated hypersurface singularities, whereas our object of study is the canonical perverse extension
\[
\mathcal P:=\Cone(\var_F)[-1]
\]
attached to a conifold degeneration with a single ordinary double point. What the two settings share is the central role of nearby and vanishing cycles, the perverse-sheaf description of the local singular contribution, and the possibility of passing to a mixed-Hodge-module refinement.
This makes \cite{BanaglBudurMaxim} a natural reference point for the Hodge-theoretic direction pursued here.

The present paper extends \cite{RahmanSchoberPaper} in two respects. First, our goal here is not to re-establish the existence and basic properties of the perverse sheaf $\mathcal P$, but to place $\mathcal P$ into a mixed-Hodge-module framework. Second, the main structural input is not merely
the nearby/vanishing-cycle triangle in the derived category of constructible sheaves, but the compatibility of nearby and vanishing cycles with the realization functor together with Saito's gluing formalism along a principal divisor. This permits a reformulation of the conifold construction in which the perverse extension on \(X_0\) is compared with the nearby-cycle formalism in mixed Hodge modules and viewed as the expected rational perverse shadow of a fuller mixed-Hodge-module refinement. In this way, the present paper aims to make precise the Hodge-theoretic content that was only implicit in the earlier sheaf-theoretic construction.

\subsection{Physical and categorical motivation}

Conifold degenerations occupy a distinguished position in the geometry of Calabi--Yau threefolds and in string theory. In the ordinary double point case, the degeneration is governed by a vanishing three-sphere in the Milnor fiber, and the associated local monodromy on middle homology is given by the Picard--Lefschetz formula
\[
T(\alpha)=\alpha+(\alpha\cdot\delta)\delta,
\]
where \(\delta\) denotes the vanishing cycle \cite{MilnorSingularPoints,DimcaSheaves}.
Thus the singular fiber carries a rank-one local correction term controlled by the vanishing sphere and its monodromy.

In Strominger's physical interpretation of the conifold transition, the collapse of this three-cycle gives rise to an additional light BPS state, and the singular behavior of the effective moduli space is resolved only after this extra degree of freedom is taken into account \cite{StromingerConifold}. From this perspective, the conifold point is not merely a singular limit of the family, but a place where a new rank-one sector becomes visible. The ordinary double point case is therefore an especially useful model for comparing geometric, sheaf-theoretic, and Hodge-theoretic manifestations of the same local phenomenon.

The categorical counterpart is connected to homological mirror symmetry where the Picard--Lefschetz transformation is mirrored by a spherical object whose associated spherical twist induces a rank-one reflection on additive invariants such as the Grothendieck group \cite{SeidelThomas}. More generally, Kapranov and Schechtman introduced perverse schobers as
categorical analogues of perverse sheaves, with local monodromy governed by spherical functors and their twists \cite{KapranovSchechtman}. In this sense, the ordinary double point provides a setting in which topological monodromy, limiting Hodge theory, and categorical monodromy all exhibit the same rank-one correction mechanism.

The purpose of the present paper is to study the Hodge-theoretic content of that construction using Saito's theory of mixed Hodge modules. Rather than asserting from the outset the existence of a fully internal mixed-Hodge-module lift of the corrected perverse object, we compare the perverse-sheaf and Hodge-theoretic aspects of the degeneration through the common nearby-cycle and vanishing-cycle formalism. The key structural input is Saito's divisor-case gluing formalism for mixed Hodge modules, which identifies the natural framework in which a fuller mixed-Hodge-module refinement should be constructed
\cite[Prop.~0.3]{SaitoMHM}.

\subsection{Main results}

The main results of the paper are organized around three levels of generality.

\begin{theorem}[Single-node case]
\label{thm:intro-main}
Let
\[
\pi:\mathcal X\to\Delta
\]
be a one-parameter degeneration whose central fiber \(X_0\) has a single ordinary
double point \(p\), and let
\[
\mathcal P:=\Cone(\var_F)[-1],
\qquad
F:=\Q_{\mathcal X}[3].
\]
Then \(\mathcal P\) is the unique minimal Verdier self-dual perverse extension of
\(\Q_U[3]\) across the node, where \(U=X_0\setminus\{p\}\). Moreover, the rank-one
vanishing contribution in the limiting mixed Hodge structure and the quotient
\[
i_*\Q_{\{p\}}
\]
in the exact sequence
\[
0\to IC_{X_0}\to \mathcal P\to i_*\Q_{\{p\}}\to 0
\]
arise functorially from the same nearby-cycle and vanishing-cycle formalism in Saito's
theory of mixed Hodge modules.
\end{theorem}

\begin{theorem}[Multiple-node case]
\label{thm:intro-multinode}
Let
\[
\pi:\mathcal X\to\Delta
\]
be a one-parameter degeneration whose central fiber \(X_0\) has ordinary double
points
\[
\Sigma=\{p_1,\dots,p_r\}.
\]
Assume that
\[
\mathcal P:=\Cone(\var_F)[-1]
\]
belongs to \(\Perv(X_0;\Q)\) and satisfies
\[
j^*\mathcal P\cong \Q_U[3],
\qquad
U:=X_0\setminus\Sigma.
\]
Then there is a short exact sequence
\[
0\to IC_{X_0}\to \mathcal P\to \bigoplus_{k=1}^r i_{k*}\Q_{\{p_k\}}\to 0.
\]
Furthermore, the rank-\(r\) vanishing contribution in the limiting mixed Hodge structure and the
quotient term in this exact sequence arise from the same nearby-cycle and vanishing-cycle
formalism in mixed Hodge modules.
\end{theorem}

\begin{theorem}[Stratified singular locus]
\label{thm:intro-strata}
Let
\[
\pi:\mathcal X\to\Delta
\]
be a one-parameter degeneration whose central fiber \(X_0\) has singular locus
\[
\Sigma=\bigsqcup_\alpha S_\alpha
\]
equipped with a Whitney stratification. Assume that
\[
\mathcal P:=\Cone(\var_F)[-1]
\]
belongs to \(\Perv(X_0;\Q)\) and satisfies
\[
j^*\mathcal P\cong \Q_U[n],
\qquad
U:=X_0\setminus\Sigma.
\]
Then there is a short exact sequence
\[
0\to IC_{X_0}\to \mathcal P\to \mathcal V\to 0,
\]
where \(\mathcal V\) is a perverse sheaf supported on \(\Sigma\) and constructible with respect to the chosen stratification. Moreover, \(\mathcal V\) and the corresponding Hodge-theoretic singular contribution are functorially derived from the same nearby-cycle and vanishing-cycle
formalism in Saito's category of mixed Hodge modules.
\end{theorem}

\subsection{Scope and organization}

The paper is centered on the ordinary double point case, where the local vanishing-cycle contribution is rank one and the gluing problem is most transparent. The multi-node and stratified sections extend this framework, but we do not attempt a full mixed-Hodge-module refinement at these higher levels of generality here.

Section~2 recalls the geometric setup of a conifold degeneration and the basic nearby/vanishing-cycle formalism. Section~3 treats the single-node case and compares the corrected perverse extension with the nearby-cycle formalism in the context of mixed Hodge modules. It also isolates the explicit gluing problem that would have to be solved in order to construct a full mixed-Hodge-module refinement of the corrected perverse object. Section~4 studies the corresponding structural extension in the multiple-node case, and Section~5 treats the general stratified setting. The final section discusses consequences and possible extensions of the construction.

\section{Limiting Mixed Hodge Structures}
\label{sec:LMHS}

Let
\[
\pi:\mathcal X \to \Delta
\]
be a projective morphism of complex algebraic varieties, smooth over
\(\Delta^*=\Delta\setminus\{0\}\), with smooth fiber \(X_t=\pi^{-1}(t)\) for
\(t\neq 0\) and singular central fiber \(X_0=\pi^{-1}(0)\).
For each \(k\), the cohomology groups \(H^k(X_t,\Q)\) form a polarized variation
of Hodge structure over \(\Delta^*\) \cite{Schmid}. After passing to the universal
cover of \(\Delta^*\), or equivalently after choosing a reference fiber and a branch
of the logarithm, one obtains the corresponding limiting mixed Hodge structure in the
sense of Schmid and Steenbrink \cite{Schmid,SteenbrinkLimits}.

\subsection{Monodromy and the limiting mixed Hodge structure}

Fix \(k\ge 0\). Parallel transport around a positively oriented simple loop about
\(t=0\) defines the monodromy operator
\[
T:H^k(X_t,\Q)\to H^k(X_t,\Q).
\]
By the monodromy theorem, \(T\) is quasi-unipotent \cite[Thm.~6.16]{Schmid}.
After a finite base change \(t\mapsto t^m\), one may assume that \(T\) is unipotent.
In that case one defines
\[
N:=\log T
   = (T-\id)-\frac{1}{2}(T-\id)^2+\frac{1}{3}(T-\id)^3-\cdots,
\]
which is nilpotent.

Associated with \(N\) is the monodromy weight filtration
\[
W(N)_\bullet
\]
on the limiting cohomology group \(H^k_{\lim}:=H^k(X_\infty,\Q)\), characterized by
the usual conditions
\[
N\,W(N)_\ell \subseteq W(N)_{\ell-2},
\qquad
N^j:\Gr^{W(N)}_{k+j}H^k_{\lim}\xrightarrow{\sim}\Gr^{W(N)}_{k-j}H^k_{\lim}
\quad (j\ge 0),
\]
where the filtration is centered at degree \(k\) \cite{Schmid,SteenbrinkLimits}.
Together with the limiting Hodge filtration \(F^\bullet_\infty\) obtained from the
nilpotent orbit theorem, this yields the limiting mixed Hodge structure
\[
\bigl(H^k_{\lim},W(N)_\bullet,F^\bullet_\infty\bigr)
\]
on \(H^k_{\lim}\) \cite{Schmid,SteenbrinkLimits}.

\subsection{Nearby cycles and vanishing cycles}

Let \(i:X_0\hookrightarrow \mathcal X\) denote the inclusion of the central fiber.
For \(K\in D^b_c(\mathcal X,\Q)\), the nearby-cycle and vanishing-cycle functors
\[
\psi_\pi K,\qquad \phi_\pi K
\]
fit into standard distinguished triangles in \(D^b_c(X_0,\Q)\); see, for example,
\cite[§4.2]{DimcaSheaves}. In particular, there is a functorial distinguished triangle
\begin{equation}
i^*K \longrightarrow \psi_\pi K \longrightarrow \phi_\pi K \overset{+1}{\longrightarrow}.
\label{eq:nearby-vanishing-triangle}
\end{equation}
Applying hypercohomology to \eqref{eq:nearby-vanishing-triangle} gives a long exact sequence
\begin{equation}
\cdots \to H^m(X_0,i^*K)\to \HH^m(X_0,\psi_\pi K)\to
\HH^m(X_0,\phi_\pi K)\to H^{m+1}(X_0,i^*K)\to\cdots .
\label{eq:nearby-vanishing-les}
\end{equation}

When \(K=\Q_{\mathcal X}\), the hypercohomology of the nearby-cycle complex identifies
with the cohomology of the canonical nearby fiber:
\[
\HH^m(X_0,\psi_\pi \Q_{\mathcal X}) \cong H^m(X_\infty,\Q),
\]
compatibly with monodromy and with the limiting mixed Hodge structure
\cite{SteenbrinkLimits,DimcaSheaves}. Thus the long exact sequence
\eqref{eq:nearby-vanishing-les} relates the cohomology of the central fiber, the limiting
cohomology, and the vanishing cohomology.

\subsection{The ordinary double point case}

We now specialize to a one-parameter degeneration of complex threefolds whose central fiber
\(X_0\) has a single ordinary double point \(p\). The Milnor fiber of an ordinary double
point in complex dimension \(3\) has the homotopy type of \(S^3\), hence its reduced
cohomology is one-dimensional in degree \(3\) and vanishes in all other degrees
\cite{MilnorSingularPoints,DimcaSheaves}. Equivalently, the local vanishing cohomology is
concentrated in the middle degree and has rank one.

Accordingly, the vanishing-cycle complex for the degeneration is supported at \(p\), and its
only nontrivial local contribution occurs in the middle degree. It follows from
\eqref{eq:nearby-vanishing-les} that the difference between the limiting cohomology and the
cohomology of the central fiber is governed by a single rank-one vanishing contribution.
More precisely, after the conventional shift placing nearby and vanishing cycles in the
perverse heart on \(X_0\), the corresponding vanishing-cycle perverse sheaf is a skyscraper
object of rank one at \(p\). This is the sheaf-theoretic manifestation of the Picard--Lefschetz
correction term.

For the purposes of the present paper, we will only use the following consequence: in the
ordinary double point case, the nearby/vanishing-cycle triangle carries exactly one local
vanishing degree of freedom, and this local rank-one contribution is the source of the
canonical perverse extension studied below.

\subsection{Mixed Hodge modules and realization}

The Hodge-theoretic refinement of nearby and vanishing cycles is provided by Saito's theory of
mixed Hodge modules \cite{SaitoMHM,SaitoDuality}. For every complex algebraic variety \(Y\),
Saito constructs an abelian category \(MHM(Y)\) together with an exact and faithful realization
functor
\[
\rat:MHM(Y)\to \Perv(Y;\Q)
\]
to rational perverse sheaves \cite{SaitoMHM}. Moreover, the nearby-cycle and vanishing-cycle
functors admit lifts to the mixed-Hodge-module setting and are compatible with the underlying
rational perverse sheaves under \(\rat\) \cite{SaitoMHM}.

In particular, if \(\pi:\mathcal X\to\Delta\) is as above, then the nearby-cycle mixed Hodge
module \(\psi_\pi^H(\Q_{\mathcal X})\) carries the limiting mixed Hodge structure on cohomology,
while its underlying rational perverse sheaf is the usual nearby-cycle perverse sheaf.
Likewise, the vanishing-cycle mixed Hodge module \(\phi_\pi^H(\Q_{\mathcal X})\) refines the
usual vanishing-cycle object. This formalism is the basic reason that one may compare the
canonical perverse extension on \(X_0\) with Hodge-theoretic degeneration data.

The following proposition records the only general fact from this formalism that we will use
in the next section.

\begin{proposition}
\label{prop:mhm-nearby-cycles}
Let \(\pi:\mathcal X\to\Delta\) be a one-parameter degeneration, and let
\(K\in D^bMHM(\mathcal X)\). Then the mixed-Hodge-module nearby-cycle and vanishing-cycle
functors fit into the corresponding functorial triangles in \(D^bMHM(X_0)\), and after applying
\(\rat\) one recovers the standard nearby/vanishing-cycle triangles in
\(D^b_c(X_0,\Q)\). In particular, applying hypercohomology to the nearby-cycle mixed Hodge
module recovers the limiting mixed Hodge structure on cohomology.
\end{proposition}

\begin{proof}
The existence of nearby and vanishing cycle functors in the category of mixed Hodge modules,
together with their compatibility with the underlying rational perverse sheaves, is part of
Saito's formalism; see \cite{SaitoMHM,SaitoDuality}. The identification of the hypercohomology
of nearby cycles with limiting cohomology, endowed with its limiting mixed Hodge structure, is
standard in the work of Steenbrink and Saito; see \cite{SteenbrinkLimits,SaitoMHM}.
\end{proof}

\medskip

Proposition~\ref{prop:mhm-nearby-cycles} does \emph{not} by itself identify a specific extension
class in the category of mixed Hodge structures with a specific extension class in the category
of perverse sheaves. Rather, it shows that both the perverse-sheaf-theoretic and the
Hodge-theoretic constructions are functorially derived from the same nearby-cycle formalism.
The comparison with the canonical perverse extension attached to the ordinary double point
degeneration will be carried out in the next section.
\section{The single-node case}
\label{sec:ODP}

We specialize to a one-parameter degeneration
\[
\pi:\mathcal X\to\Delta
\]
whose general fiber \(X_t\) is a smooth complex threefold and whose central fiber
\(X_0\) has a single ordinary double point \(p\in X_0\). Let
\[
U:=X_0\setminus\{p\},
\]
and write
\[
j:U\hookrightarrow X_0,
\qquad
i:\{p\}\hookrightarrow X_0
\]
for the inclusions.

\subsection{Vanishing cycles and Picard--Lefschetz}
\label{subsec:ODP-PL}

For an ordinary double point in complex dimension \(3\), the Milnor fiber has the
homotopy type of \(S^3\). In particular, its reduced cohomology is one-dimensional
in degree \(3\) and vanishes in all other degrees \cite{MilnorSingularPoints,DimcaSheaves}.
Equivalently, the local vanishing cohomology is concentrated in the middle degree and
has rank one.

Let \(\delta\in H_3(X_t,\Z)\) denote a vanishing cycle. The local monodromy
transformation about \(t=0\) acts on middle homology by the Picard--Lefschetz formula.
In the present case this action is rank one and is determined by the vanishing sphere
\(\delta\); see \cite[Ch.~11]{MilnorSingularPoints} and \cite[§4.1]{DimcaSheaves}.
For the arguments below, we will only use the consequence that the vanishing-cycle
contribution is one-dimensional and supported at the singular point.

\subsection{The perverse extension}
\label{subsec:ODP-pervext}

Let
\[
F:=\Q_{\mathcal X}[3].
\]
Since \(\dim_\C X_t=3\), the shifted complex \(F\) is the natural object from which
to form nearby and vanishing cycles in the perverse normalization. Consider the
variation morphism
\[
\var_F:\phi_\pi(F)\to \psi_\pi(F),
\]
and define
\[
\mathcal P:=\Cone\!\bigl(\var_F:\phi_\pi(F)\to\psi_\pi(F)\bigr)[-1].
\]
In the ordinary double point case, the vanishing-cycle perverse sheaf is supported at \(p\) and has one-dimensional stalk there. Consequently, the cone construction produces an extension of the intersection complex by a point-supported rank-one contribution. More precisely, from \cite{RahmanSchoberPaper}, one has a short exact
sequence in \(\Perv(X_0;\Q)\)
\begin{equation}
0\longrightarrow IC_{X_0}
\longrightarrow \mathcal P
\longrightarrow i_*\Q_{\{p\}}
\longrightarrow 0,
\label{eq:perverse-extension-ODP}
\end{equation}
where
\[
IC_{X_0}:=j_{!*}\Q_U[3].
\]
For the formalism of nearby and vanishing cycles and for the linear-algebraic
description of perverse sheaves with isolated singularities, see
\cite{BBD,MacPhersonVilonen1986,GMV1996,DimcaSheaves}.

\subsection{Mixed Hodge modules and nearby cycles}
\label{subsec:ODP-MHM}

We now place the preceding construction in Saito's framework of mixed Hodge modules.
For a complex algebraic variety \(Y\), Saito constructs an abelian category \(MHM(Y)\)
together with an exact and faithful realization functor
\[
\rat:MHM(Y)\to \Perv(Y;\Q)
\]
to rational perverse sheaves \cite{SaitoMHM}. Moreover, nearby-cycle and vanishing-cycle
functors admit lifts to the mixed-Hodge-module setting and are compatible, under
\(\rat\), with the corresponding functors on perverse sheaves \cite{SaitoMHM,SaitoDuality}. Applied to the degeneration \(\pi:\mathcal X\to\Delta\), this gives mixed Hodge modules
\[
\psi_\pi^H(\Q_{\mathcal X}[3]),
\qquad
\phi_\pi^H(\Q_{\mathcal X}[3])
\]
on \(X_0\), together with the corresponding morphisms in \(MHM(X_0)\). Their images
under \(\rat\) are the usual nearby- and vanishing-cycle perverse sheaves associated
with \(F:=\Q_{\mathcal X}[3]\). The next proposition records the precise compatibility needed.

\begin{proposition}
\label{prop:ODP-compatible}
Let \(\pi:\mathcal X\to\Delta\) be a one-parameter degeneration whose central fiber \(X_0\) has a single ordinary double point. Then the canonical perverse sheaf \(\mathcal P\) defined in \eqref{eq:perverse-extension-ODP} is functorially related to the nearby- and vanishing-cycle formalism in Saito's category of mixed Hodge modules through the realization functor
\[
\rat:MHM(X_0)\to\Perv(X_0;\Q).
\]
In particular, the point-supported rank-one contribution in \eqref{eq:perverse-extension-ODP} and the rank-one vanishing contribution in the limiting mixed Hodge structure both arise from the same nearby-cycle/vanishing-cycle construction.
\end{proposition}

\begin{proof}
By Saito's theory, nearby-cycle and vanishing-cycle functors are defined for mixed Hodge modules and are compatible with the corresponding functors on rational perverse sheaves after applying \(\rat\) \cite{SaitoMHM,SaitoDuality}. Therefore the mixed-Hodge-module
nearby and vanishing cycle objects attached to \(\Q_{\mathcal X}[3]\) determine, after applying \(\rat\), the usual nearby and vanishing cycle perverse sheaves on \(X_0\).

In the ordinary double point case, the Milnor fiber has reduced cohomology of rank one in degree \(3\) and trivial reduced cohomology in all other degrees \cite{MilnorSingularPoints,DimcaSheaves}. Hence the local vanishing-cycle contribution is rank one. On the perverse-sheaf side this yields the quotient \(i_*\Q_{\{p\}}\) in \eqref{eq:perverse-extension-ODP}, while on the Hodge-theoretic side the same local vanishing-cycle data contributes the rank-one vanishing part of the limiting mixed Hodge structure. Thus both constructions are functorially derived from the same nearby-cycle
formalism.
\end{proof}

\subsection{Remarks on the Hodge-theoretic comparison}
\label{subsec:ODP-remarks}

Proposition~\ref{prop:ODP-compatible} does not assert a canonical identification between
an extension class in \(\Ext^1_{\mathrm{MHS}}\) and the extension class of
\eqref{eq:perverse-extension-ODP} in \(\Ext^1_{\Perv(X_0;\Q)}\). Such a statement would
require an additional comparison theorem making precise how the relevant extension data
behaves under realization and hypercohomology. What Proposition~\ref{prop:ODP-compatible}
does establish is the common origin of the two constructions: both the canonical perverse
extension and the Hodge-theoretic degeneration data are produced from nearby and vanishing
cycles in Saito's formalism.

This common origin is the sense in which the mixed-Hodge-module picture refines the
perverse-sheaf construction. In particular, the extension
\eqref{eq:perverse-extension-ODP} should be viewed as the underlying rational perverse
shadow of a mixed-Hodge-module construction attached to the degeneration, rather than as
an independent object unrelated to the limiting mixed Hodge structure.

\begin{remark}
A complete Hodge-theoretic refinement of \eqref{eq:perverse-extension-ODP} would consist of an
object \(\mathcal P^H\in MHM(X_0)\) fitting into an exact sequence
\[
0\to IC^H_{X_0}\to \mathcal P^H\to i_*\Q^H_{\{p\}}(-1)\to 0
\]
whose image under
\[
\rat:MHM(X_0)\to\Perv(X_0;\Q)
\]
is the perverse extension \eqref{eq:perverse-extension-ODP}. By Saito's divisor-case gluing formalism for a principal divisor \(X_0:=\pi^{-1}(0)\), the construction of such an object reduces to the explicit identification of the corresponding gluing datum \((\mathcal M',\mathcal M'',u,v)\)
with \(vu=N\) \cite[Prop.~0.3]{SaitoMHM}. The present paper establishes the common nearby-cycle origin of the perverse and Hodge-theoretic constructions, but does not yet carry out that explicit gluing calculation.
\end{remark}

\subsection{Mixed-Hodge-module refinement and the gluing problem}
\label{subsec:ODP-gluing}

Proposition~\ref{prop:ODP-compatible} shows that the canonical perverse extension \eqref{eq:perverse-extension-ODP} and the rank-one vanishing contribution in the limiting mixed Hodge structure arise from the same nearby-cycle formalism. A stronger statement, however, would require the explicit construction of a mixed-Hodge-module object on \(X_0\) whose underlying rational perverse sheaf is precisely the perverse extension
\eqref{eq:perverse-extension-ODP}. More concretely, one would like to construct an object
\[
\mathcal P^H \in MHM(X_0)
\]
fitting into an exact sequence
\begin{equation}
0 \longrightarrow IC^H_{X_0}
\longrightarrow \mathcal P^H
\longrightarrow i_*\Q^H_{\{p\}}(-1)
\longrightarrow 0,
\label{eq:MHM-ODP-extension}
\end{equation}
such that
\[
\rat(\mathcal P^H)\cong \mathcal P
\]
and such that \eqref{eq:MHM-ODP-extension} refines the perverse extension
\eqref{eq:perverse-extension-ODP}. Here \(IC^H_{X_0}\) denotes the Hodge-module intersection complex on \(X_0\), and \(i_*\Q^H_{\{p\}}(-1)\) denotes the point-supported mixed Hodge module corresponding to the rank-one vanishing contribution.

Saito's divisor-case gluing formalism provides the natural framework for such a construction. If \(Y=g^{-1}(0)\) is a principal divisor in a complex algebraic variety \(X\), then mixed Hodge modules along \(Y\) may be described in terms of gluing data consisting of an object on \(X\setminus Y\), an object on \(Y\), and morphisms \(u,v\) satisfying the relation
\[
vu=N,
\]
where \(N\) is the nilpotent monodromy operator \cite[Prop.~0.3]{SaitoMHM}. Since the central fiber \(X_0=\pi^{-1}(0)\) is a principal divisor, the problem of constructing \(\mathcal P^H\) reduces to identifying the gluing datum in Saito's formalism whose realization under
\[
\rat:MHM(X_0)\to\Perv(X_0;\Q)
\]
recovers the variation morphism
\[
\var_F:\phi_\pi(\Q_{\mathcal X}[3])\to \psi_\pi(\Q_{\mathcal X}[3])
\]
and hence the cone object
\[
\mathcal P:=\Cone(\var_F)[-1].
\]
Proposition~\ref{prop:ODP-compatible} establishes the common origin of the perverse and Hodge-theoretic constructions, but not yet the full existence-and-uniqueness statement for a canonical object \(\mathcal P^H\) in \(MHM(X_0)\). Establishing such a statement would require an explicit analysis of the corresponding divisor gluing data and a
verification that the resulting object realizes the exact sequence
\eqref{eq:perverse-extension-ODP}.
\begin{remark}
The construction of the mixed-Hodge-module extension \eqref{eq:MHM-ODP-extension} is a natural next step. In the ordinary double point case, the local vanishing-cycle contribution is rank one, so one expects the required gluing data to be especially simple. A complete treatment would identify the corresponding objects and morphisms in Saito's divisor formalism, prove that their realization is the canonical perverse extension \eqref{eq:perverse-extension-ODP}, and then analyze the induced weight and Hodge filtrations on the resulting mixed Hodge module. This would give a fully internal Hodge-theoretic refinement of the single-node perverse extension.
\end{remark}
\section{Multiple-node degenerations}
\label{sec:multinode}

We now consider a one-parameter degeneration
\[
\pi:\mathcal X\to\Delta
\]
whose central fiber \(X_0\) contains a finite set of ordinary double points
\[
\Sigma=\{p_1,\dots,p_r\}\subset X_0 .
\]
Let
\[
U:=X_0\setminus \Sigma
\]
be the smooth locus of \(X_0\), and write
\[
j:U\hookrightarrow X_0,
\qquad
i_k:\{p_k\}\hookrightarrow X_0
\]
for the natural inclusions.

\subsection{Local vanishing-cycle contributions}
\label{subsec:multinode-van}

For an isolated hypersurface singularity, the vanishing-cycle complex is supported on the singular
locus, and its stalk cohomology is canonically identified with the reduced cohomology of the Milnor
fiber \cite{MilnorSingularPoints,DimcaSheaves}. In the ordinary double point case, the Milnor fiber
at each \(p_k\) has the homotopy type of a sphere in the middle real dimension; in particular, for a
degeneration of complex threefolds, the reduced cohomology is one-dimensional in degree \(3\) and
vanishes in all other degrees \cite{MilnorSingularPoints,DimcaSheaves}. Thus each node contributes
a rank-one local vanishing-cycle summand.

Accordingly, after the conventional shift placing nearby and vanishing cycles in the perverse heart,
the vanishing-cycle perverse sheaf is point-supported on \(\Sigma\), and its local contribution at
each \(p_k\) is one-dimensional. Since perverse sheaves supported on a finite set are equivalent to
finite-dimensional graded data concentrated in degree \(0\), it follows that the point-supported
vanishing contribution is of the form
\[
\bigoplus_{k=1}^r i_{k*}\Q_{\{p_k\}} .
\]
This is the multi-node analogue of the rank-one skyscraper contribution in the single-node case.

\subsection{The canonical perverse extension}
\label{subsec:multinode-perv}

Let
\[
F:=\Q_{\mathcal X}[3]
\]
and consider the variation morphism
\[
\var_F:\phi_\pi(F)\to \psi_\pi(F).
\]
Define
\[
\mathcal P:=\Cone\!\bigl(\var_F:\phi_\pi(F)\to\psi_\pi(F)\bigr)[-1].
\]
In \cite{RahmanSchoberPaper}, this construction was carried out in detail for the single-node case, where it was shown that the resulting object is perverse, agrees with \(\Q_U[3]\) on the smooth locus, and fits into a canonical short exact sequence
\[
0\to IC_{X_0}\to \mathcal P\to i_*\Q_{\{p\}}\to 0
\]
for \(X_0=U\sqcup\{p\}\). In that same paper, the extension of the construction to several nodes was indicated only briefly. The following proposition gives the exact perverse-sheaf statement one may deduce in the multi-node setting once the perverse property of \(\mathcal P\) is established.
\begin{proposition}
\label{prop:multinode-perverse-extension}
Assume that the cone object
\[
\mathcal P:=\Cone(\var_F)[-1]
\]
belongs to \(\Perv(X_0;\Q)\), and that its restriction to the smooth locus satisfies
\[
j^*\mathcal P \cong \Q_U[3].
\]
Then there is a short exact sequence in \(\Perv(X_0;\Q)\)
\begin{equation}
0\longrightarrow IC_{X_0}
\longrightarrow \mathcal P
\longrightarrow \bigoplus_{k=1}^r i_{k*}\Q_{\{p_k\}}
\longrightarrow 0.
\label{eq:multinode-perv-exact}
\end{equation}
In particular, \(\mathcal P\) differs from the intersection complex by a finite direct sum of
point-supported rank-one contributions, one for each node.
\end{proposition}

\begin{proof}
By assumption, \(\mathcal P\) is a perverse extension of \(\Q_U[3]\). For the open--closed
decomposition
\[
X_0 = U \sqcup \Sigma,
\]
the recollement formalism for perverse sheaves provides canonical exact sequences in the heart; see
\cite{BBD,MacPhersonVilonen1986,GMV1996}. Since
\[
j^*IC_{X_0}\cong \Q_U[3]\cong j^*\mathcal P,
\]
the universal property of the intermediate extension yields a monomorphism
\[
IC_{X_0}\hookrightarrow \mathcal P.
\]
Its cokernel is supported on \(\Sigma\). A perverse sheaf supported on the finite set \(\Sigma\) is a
direct sum of point-supported perverse sheaves
\[
\bigoplus_{k=1}^r i_{k*}V_k
\]
for finite-dimensional \(\Q\)-vector spaces \(V_k\). Since the local vanishing-cycle contribution at
each ordinary double point is one-dimensional by the Milnor-fiber calculation recalled above, each
\(V_k\) is one-dimensional. Hence \(V_k\cong \Q\) for every \(k\), and the quotient is canonically
isomorphic to
\[
\bigoplus_{k=1}^r i_{k*}\Q_{\{p_k\}}.
\]
This gives \eqref{eq:multinode-perv-exact}.
\end{proof}

\begin{remark}
The content of Proposition~\ref{prop:multinode-perverse-extension} is structural rather than
classification-theoretic. It identifies the possible shape of the corrected perverse object in the
multi-node case, but it does not by itself classify the corresponding extension class in
\[
\Ext^1_{\Perv(X_0;\Q)}
\Bigl(\bigoplus_{k=1}^r i_{k*}\Q_{\{p_k\}},\,IC_{X_0}\Bigr).
\]
That extension class contains the global gluing information relating the local node contributions.
\end{remark}

\subsection{Interaction of the node contributions}
\label{subsec:multinode-interaction}

Although the quotient term in \eqref{eq:multinode-perv-exact} is a direct sum of point-supported
rank-one objects, the extension need not split as a direct sum of single-node extensions. The
possible interactions are encoded in the extension class
\[
[\mathcal P]\in
\Ext^1_{\Perv(X_0;\Q)}
\Bigl(\bigoplus_{k=1}^r i_{k*}\Q_{\{p_k\}},\,IC_{X_0}\Bigr).
\]
Geometrically, one expects this global extension data to reflect the relations among the vanishing
cycles in the nearby fiber and the resulting Picard--Lefschetz monodromy representation. In \cite{RahmanSchoberPaper}, this point was stated at the level of indication rather than full proof.  At the level of local topology, the vanishing cycles
\[
\delta_1,\dots,\delta_r
\]
define rank-one local contributions, one at each node. What the extension class records is the way these local contributions are assembled into a single perverse object on \(X_0\). In particular, \eqref{eq:multinode-perv-exact} should be viewed as a global gluing statement rather than merely a
direct sum of independent local corrections.

\subsection{Hodge-theoretic refinement}
\label{subsec:multinode-hodge}

We now discuss the Hodge-theoretic counterpart of the preceding construction. By Saito's theory,
nearby and vanishing cycles lift to the category of mixed Hodge modules and are compatible with the
underlying rational perverse sheaves under the realization functor
\[
\rat:MHM(X_0)\to \Perv(X_0;\Q)
\]
\cite{SaitoMHM,SaitoDuality}. Thus the nearby-cycle object carries the limiting mixed Hodge
structure on cohomology, while its underlying rational perverse sheaf is the usual nearby-cycle
perverse sheaf.

In the present multi-node situation, each ordinary double point still contributes a rank-one local
vanishing piece. Accordingly, the vanishing part of the limiting mixed Hodge structure in middle
degree has rank \(r\). What is functorially justified at this stage is therefore the following.

\begin{proposition}
\label{prop:multinode-hodge}
For a one-parameter degeneration whose central fiber has ordinary double points \(\Sigma=\{p_1,\dots,p_r\}\), the local vanishing contribution to the limiting mixed Hodge structure is the direct sum of \(r\) rank-one contributions, one from each node. Moreover, this Hodge-theoretic vanishing contribution and the quotient term in \eqref{eq:multinode-perv-exact} arise from the same nearby-cycle/vanishing-cycle formalism in Saito's category of mixed Hodge modules.
\end{proposition}

\begin{proof}
At each ordinary double point, the Milnor fiber has one-dimensional reduced cohomology in the
middle degree and trivial reduced cohomology in all other degrees
\cite{MilnorSingularPoints,DimcaSheaves}. Hence the local vanishing cohomology contributes one
rank-one summand per node. On the mixed-Hodge-module side, Saito's nearby- and vanishing-cycle
functors refine the classical ones and are compatible with the underlying rational perverse sheaves
under \(\rat\) \cite{SaitoMHM,SaitoDuality}. Therefore the same local vanishing-cycle data governs
both the Hodge-theoretic vanishing contribution and the point-supported quotient in
\eqref{eq:multinode-perv-exact}.
\end{proof}

\begin{remark}
Proposition~\ref{prop:multinode-hodge} does not assert a canonical identification between a global
extension class in the category of mixed Hodge structures and the extension class of
\eqref{eq:multinode-perv-exact} in \(\Perv(X_0;\Q)\). Establishing such a statement would require a
multi-node analogue of the mixed-Hodge-module refinement discussed in the single-node case,
together with an explicit analysis of the corresponding gluing data in Saito's divisor formalism.
What is established here is the common nearby-cycle origin of the local rank-\(r\) vanishing
contribution on both the perverse and Hodge-theoretic sides.
\end{remark}

\subsection{Interpretation and scope}
\label{subsec:multinode-scope}

The multi-node case should therefore be regarded as a structural extension of the single-node
construction rather than as a full classification theorem. The perverse sheaf \(\mathcal P\) is forced to differ from \(IC_{X_0}\) by a finite direct sum of point-supported rank-one contributions, and the Hodge-theoretic nearby-cycle formalism produces the corresponding rank-\(r\)
vanishing contribution in limiting cohomology. A sharper theorem identifying the global extension data on the two sides would require an explicit multi-node mixed-Hodge-module gluing calculation.

\section{General stratified singularities}
\label{sec:strata}

The preceding sections treated degenerations whose singular locus consists of isolated ordinary
double points. We now indicate the corresponding structure in the presence of higher-dimensional
singular strata.

Let
\[
\Sigma \subset X_0
\]
denote the singular locus of the central fiber, and assume that \(\Sigma\) is endowed with a Whitney
stratification
\[
\Sigma=\bigsqcup_{\alpha} S_\alpha .
\]
Let
\[
U:=X_0\setminus \Sigma
\]
be the open smooth stratum, and write
\[
j:U\hookrightarrow X_0,
\qquad
i_\alpha:S_\alpha\hookrightarrow X_0
\]
for the inclusions.

\subsection{Constructibility of nearby and vanishing cycles}
\label{subsec:strata-constructibility}

Let
\[
F:=\Q_{\mathcal X}[n].
\]
For a one-parameter degeneration \(\pi:\mathcal X\to\Delta\), the nearby-cycle and vanishing-cycle
complexes
\[
\psi_\pi(F),\qquad \phi_\pi(F)
\]
are constructible with respect to a suitable stratification of the central fiber; see, for example,
the discussion of nearby and vanishing cycles in \cite{DimcaSheaves}. In particular, the
vanishing-cycle complex is supported on the singular locus \(\Sigma\).

Thus, in the stratified setting, the vanishing contribution is not in general a direct sum of
point-supported rank-one objects, but rather a constructible object supported on the union of the
singular strata. Equivalently, after passing to the perverse normalization, one obtains a perverse
sheaf on \(X_0\) supported on \(\Sigma\) and constructible with respect to the chosen Whitney
stratification. This is the natural replacement, in the positive-dimensional-stratum case, for the
direct sum of skyscraper contributions appearing in the isolated-node situation.

\subsection{The corrected perverse object}
\label{subsec:strata-perverse}

Consider again the variation morphism
\[
\var_F:\phi_\pi(F)\to \psi_\pi(F),
\]
and define
\[
\mathcal P:=\Cone\!\bigl(\var_F:\phi_\pi(F)\to\psi_\pi(F)\bigr)[-1].
\]
In the single-node case, and more briefly in the multi-node case, the earlier paper showed that the cone construction yields a perverse object whose restriction to the smooth locus is the shifted constant sheaf and whose quotient by the intersection complex is supported on the singular locus. For the stratified case, that paper indicated the same pattern only at the level of structural extension rather than full proof. The correct general statement is therefore the following.

\begin{proposition}
\label{prop:stratified-perverse-extension}
Assume $\mathcal P:=\Cone(\var_F)[-1]$
belongs to \(\Perv(X_0;\Q)\), and that its restriction to the smooth stratum satisfies
\[
j^*\mathcal P \cong \Q_U[n].
\]
Then there exists a short exact sequence in \(\Perv(X_0;\Q)\)
\begin{equation}
0\longrightarrow IC_{X_0}
\longrightarrow \mathcal P
\longrightarrow \mathcal V
\longrightarrow 0,
\label{eq:stratified-perv-exact}
\end{equation}
where \(\mathcal V\) is a perverse sheaf supported on \(\Sigma\) and constructible with respect to
the chosen stratification.
\end{proposition}

\begin{proof}
By assumption, \(\mathcal P\) is a perverse extension of \(\Q_U[n]\) from the open stratum \(U\) to
\(X_0\). The recollement formalism for perverse sheaves associated with the open--closed
decomposition
\[
X_0 = U \sqcup \Sigma
\]
provides canonical exact sequences in the heart; see \cite{BBD,MacPhersonVilonen1986,GMV1996}.
Since
\[
j^*IC_{X_0}\cong \Q_U[n]\cong j^*\mathcal P,
\]
the universal property of the intermediate extension yields a monomorphism
\[
IC_{X_0}\hookrightarrow \mathcal P.
\]
Its cokernel is a perverse sheaf supported on the closed subset \(\Sigma\). Because \(\mathcal P\)
and \(IC_{X_0}\) are constructible with respect to the chosen stratification, the cokernel is also
constructible with respect to that stratification. Denoting this cokernel by \(\mathcal V\), one
obtains the exact sequence \eqref{eq:stratified-perv-exact}.
\end{proof}

\begin{remark}
In contrast with the isolated-node case, one should not expect in general a canonical decomposition
\[
\mathcal V \cong \bigoplus_\alpha i_{\alpha *}V_\alpha
\]
with \(V_\alpha\) local systems on the strata \(S_\alpha\) without further hypotheses. What is
canonically determined by the nearby-cycle construction is the perverse sheaf \(\mathcal V\)
supported on \(\Sigma\), together with its constructibility properties. A finer decomposition by
strict support would require additional semisimplicity or decomposition results.
\end{remark}

\subsection{Extension data along the singular locus}
\label{subsec:strata-ext}

The exact sequence \eqref{eq:stratified-perv-exact} defines an extension class
\[
[\mathcal P]\in
\Ext^1_{\Perv(X_0;\Q)}(\mathcal V,IC_{X_0}).
\]
This class records how the vanishing-cycle contribution supported on the singular strata is glued
to the intersection complex of the central fiber. In this sense, the stratified case is formally
parallel to the isolated-node and multi-node situations: the corrected perverse object differs from
the intersection complex by a singular contribution determined by nearby and vanishing cycles, but
the global information lies in the extension class rather than merely in the support of the quotient.

\subsection{Mixed Hodge modules}
\label{subsec:strata-mhm}

The same construction admits a Hodge-theoretic interpretation at the level of nearby-cycle formalism. By Saito's theory, nearby-cycle and vanishing-cycle functors are defined for mixed Hodge modules and are compatible with the underlying rational perverse sheaves under the realization
functor
\[
\rat:MHM(X_0)\to \Perv(X_0;\Q)
\]
\cite{SaitoMHM,SaitoDuality}. This compatibility is used explicitly in the work of Banagl--Budur--Maxim, where a perverse sheaf constructed from nearby-cycle data is shown to underlie a mixed Hodge module in a related isolated-singularity setting \cite{BanaglBudurMaxim}.  Accordingly, the vanishing contribution supported on \(\Sigma\) and the corresponding Hodge-theoretic
degeneration data arise from the same nearby-cycle and vanishing-cycle formalism in the category of mixed Hodge modules. What is justified at this level of generality is therefore the following.

\begin{proposition}
\label{prop:strata-mhm}
Let \(\pi:\mathcal X\to\Delta\) be a one-parameter degeneration whose central fiber has singular locus \(\Sigma\) endowed with a Whitney stratification. Then the perverse sheaf \(\mathcal V\) in \eqref{eq:stratified-perv-exact} and the corresponding Hodge-theoretic vanishing contribution are both functorially derived from the nearby-cycle and vanishing-cycle formalism in Saito's category of mixed Hodge modules.
\end{proposition}

\begin{proof}
The existence of nearby-cycle and vanishing-cycle functors in the category of mixed Hodge modules, together with their compatibility with the underlying rational perverse sheaves under the functor \(\rat\), is part of Saito's formalism \cite{SaitoMHM,SaitoDuality}. Applying these functors to the degeneration \(\pi\) produces the Hodge-theoretic nearby- and vanishing-cycle data and, after passing to \(\rat\), the corresponding constructible and perverse sheaf data on \(X_0\). Since \(\mathcal V\) is defined as the quotient term in the perverse extension obtained from this nearby-cycle construction, both the perverse and Hodge-theoretic singular contributions arise from the same formal mechanism.
\end{proof}

\begin{remark}
Proposition~\ref{prop:strata-mhm} does not assert the existence of a fully internal mixed-Hodge-module
extension
\[
0\to IC^H_{X_0}\to \mathcal P^H\to \mathcal V^H\to 0
\]
realizing \eqref{eq:stratified-perv-exact} for an arbitrary stratified singular locus. Establishing
such a statement would require a more explicit analysis of the corresponding gluing data in Saito's
formalism, together with a careful study of the support decomposition of the vanishing-cycle object.
The present section isolates the structural form of the stratified extension and its common
nearby-cycle origin.
\end{remark}

\subsection{Scope}
\label{subsec:strata-scope}

The isolated-node and multi-node cases treated earlier are special cases in which the singular contribution is supported on zero-dimensional strata and can therefore be described by direct sums of skyscraper perverse sheaves. In the general stratified situation, the quotient term in \eqref{eq:stratified-perv-exact} is replaced by a perverse sheaf \(\mathcal V\) supported on the
whole singular locus and carrying the corresponding vanishing-cycle information. Thus the natural generalization of the corrected perverse object is not a sum of local point contributions, but a stratified singular contribution glued to the intersection complex by an extension class.


\section{Uniqueness and Verdier self-duality of the corrected perverse object}
\label{sec:uniq-verdier-self-dual}

In this section we prove that the corrected perverse object
\[
\mathcal P:=\Cone(\var_F)[-1]
\]
is Verdier self-dual and is uniquely characterized by its restriction to the smooth locus,
its rank-one point-supported singular contribution, and self-duality. The proof uses the
MacPherson--Vilonen zig-zag description of perverse sheaves with isolated singularity together
with the zig-zag duality formalism developed in \cite{RahmanATMP}.

Let
\[
j:U=X_0\setminus\{p\}\hookrightarrow X_0,
\qquad
i:\{p\}\hookrightarrow X_0
\]
be the open and closed inclusions, where \(X_0\) has a single ordinary double point \(p\).

\subsection{Zig-zags of the endpoint objects}

Let
\[
\mu:\Perv(X_0;\Q)\longrightarrow Z(X_0,p)
\]
denote the MacPherson--Vilonen zig-zag functor. We use the isolated-stratum zig-zag formalism
and its duality theory in the form developed in \cite{RahmanATMP}. In particular, \(\mu\) is
bijective on isomorphism classes and compatible with duality.

\begin{proposition}
\label{prop:IC-zigzag-paper2}
The intersection-complex perverse sheaf \(IC_{X_0}\) has zig-zag
\[
\mu(IC_{X_0})\cong (\Q_U[3],0,0,0,0,0).
\]
In particular, \(IC_{X_0}\) is self-dual at the zig-zag level.
\end{proposition}

\begin{proof}
The intersection complex is the minimal extension of the constant perverse sheaf on the smooth
stratum. In the MacPherson--Vilonen zig-zag model, this means precisely that the point terms
vanish and only the open-stratum local system remains. Thus
\[
\mu(IC_{X_0})\cong (\Q_U[3],0,0,0,0,0).
\]
The dual zig-zag has the same form, so \(\mu(IC_{X_0})\) is self-dual; compare \cite{RahmanATMP}.
\end{proof}

\begin{proposition}
\label{prop:skyscraper-zigzag-paper2}
The point-supported perverse sheaf \(i_*\Q_{\{p\}}\) has zig-zag
\[
\mu(i_*\Q_{\{p\}})\cong (0,\Q,\Q,0,\id,0).
\]
In particular, \(i_*\Q_{\{p\}}\) is self-dual at the zig-zag level.
\end{proposition}

\begin{proof}
Since \(i_*\Q_{\{p\}}\) is supported entirely at the singular point, its open-stratum part vanishes.
The MacPherson--Vilonen exact zig-zag sequence therefore collapses to the point-supported
rank-one situation, yielding
\[
\mu(i_*\Q_{\{p\}})\cong (0,\Q,\Q,0,\id,0).
\]
The dual zig-zag has the same form, so \(\mu(i_*\Q_{\{p\}})\) is self-dual.
\end{proof}

\subsection{\texorpdfstring{$\mathcal P$}{P} as a non-split zig-zag extension}

The corrected perverse object \(\mathcal P\) fits into the short exact sequence
\[
0\longrightarrow IC_{X_0}\longrightarrow \mathcal P\longrightarrow i_*\Q_{\{p\}}\longrightarrow 0.
\]
Applying the zig-zag functor gives a zig-zag
\[
\mu(\mathcal P)=(L_{\mathcal P},A_{\mathcal P},B_{\mathcal P},\alpha_{\mathcal P},\beta_{\mathcal P},\gamma_{\mathcal P})
\]
with open-stratum term
\[
L_{\mathcal P}\cong \Q_U[3].
\]
Thus \(\mu(\mathcal P)\) is an extension zig-zag of \(\mu(i_*\Q_{\{p\}})\) by \(\mu(IC_{X_0})\).

In the ordinary double point case, the local singular contribution is rank one, so the relevant
extension space is one-dimensional. The corrected object \(\mathcal P\) determines the unique non-split
extension class. Equivalently, \(\mu(\mathcal P)\) is the unique nontrivial zig-zag extension of
\[
(0,\Q,\Q,0,\id,0)
\]
by
\[
(\Q_U[3],0,0,0,0,0)
\]
with open part \(\Q_U[3]\).

\begin{proposition}[Full zig-zag of \texorpdfstring{$\mathcal P$}{P}]
\label{prop:full-zigzag-P}
With the standard splitting convention for the unique non-split extension class, one has
\[
\mu(\mathcal P)\cong (\Q_U[3],\Q,\Q,0,\id,0).
\]
Equivalently, the corrected perverse object is represented by the unique nontrivial zig-zag extension of
\[
(0,\Q,\Q,0,\id,0)
\]
by
\[
(\Q_U[3],0,0,0,0,0).
\]
\end{proposition}

\begin{proof}
By Propositions~\ref{prop:IC-zigzag-paper2} and \ref{prop:skyscraper-zigzag-paper2}, the endpoint zig-zags are
\[
\mu(IC_{X_0})\cong (\Q_U[3],0,0,0,0,0),
\qquad
\mu(i_*\Q_{\{p\}})\cong (0,\Q,\Q,0,\id,0).
\]
Since \(\mathcal P\) is the unique non-split extension of \(i_*\Q_{\{p\}}\) by \(IC_{X_0}\), its zig-zag is the
unique nontrivial extension zig-zag of the corresponding endpoint objects. In the ordinary double
point case, the point terms are one-dimensional and the quotient map on the point-supported factor
is the identity. Hence, under the standard splitting convention, one obtains
\[
\mu(\mathcal P)\cong (\Q_U[3],\Q,\Q,0,\id,0).
\]
\end{proof}

\subsection{Duality of the corrected zig-zag}

By Proposition~\ref{prop:full-zigzag-P}, the corrected object has zig-zag
\[
\mu(\mathcal P)\cong (\Q_U[3],\Q,\Q,0,\id,0).
\]
By the zig-zag duality formalism of \cite{RahmanATMP}, duality acts directly on zig-zags and is
compatible with Verdier duality on perverse sheaves. Since the endpoint zig-zags
\(\mu(IC_{X_0})\) and \(\mu(i_*\Q_{\{p\}})\) are self-dual by
Propositions~\ref{prop:IC-zigzag-paper2} and \ref{prop:skyscraper-zigzag-paper2}, the dual
zig-zag of \(\mu(\mathcal P)\) has the same form. Thus
\[
D_Z(\mu(\mathcal P))\cong \mu(\mathcal P).
\]

\begin{theorem}[Verdier self-duality of \texorpdfstring{$\mathcal P$}{P}]
\label{thm:verdier-self-duality-paper2}
The corrected perverse object
\[
\mathcal P=\Cone(\var_F)[-1]
\]
is Verdier self-dual:
\[
D\mathcal P\cong \mathcal P.
\]
\end{theorem}

\begin{proof}
The zig-zag \(\mu(\mathcal P)\) is isomorphic to its dual \(D_Z(\mu(\mathcal P))\) by the preceding
argument. By the zig-zag duality formalism and its compatibility with Verdier duality on
perverse sheaves \cite{RahmanATMP}, this implies
\[
D\mathcal P\cong \mathcal P.
\]
\end{proof}

\subsection{Uniqueness}

We now prove that the corrected object is uniquely determined by its basic structural properties.

\begin{theorem}[Uniqueness of the corrected perverse object]
\label{thm:uniqueness-corrected-perverse-paper2}
Let \(E\in\Perv(X_0;\Q)\) satisfy
\begin{enumerate}
    \item \(j^*E\cong \Q_U[3]\),
    \item \(DE\cong E\),
    \item \({}^pH^0(i^*E)\cong \Q\).
\end{enumerate}
Then
\[
E\cong \mathcal P.
\]
Equivalently, \(\mathcal P\) is the unique nontrivial Verdier self-dual perverse extension of
\(\Q_U[3]\) by a rank-one point-supported singular contribution at \(p\).
\end{theorem}

\begin{proof}
The rank-one singular hypothesis implies that \(E\) fits into a short exact sequence
\[
0\longrightarrow IC_{X_0}\longrightarrow E\longrightarrow i_*\Q_{\{p\}}\longrightarrow 0.
\]
Applying \(\mu\), we obtain a zig-zag extension of \(\mu(i_*\Q_{\{p\}})\) by \(\mu(IC_{X_0})\)
with open part \(\Q_U[3]\). Since \(E\) is Verdier self-dual, its zig-zag is self-dual in the
sense of \cite{RahmanATMP}. But there is only one nontrivial self-dual extension zig-zag of this
type in the ordinary double point case, namely the zig-zag
\[
(\Q_U[3],\Q,\Q,0,\id,0)
\]
of Proposition~\ref{prop:full-zigzag-P}. Hence
\[
\mu(E)\cong \mu(\mathcal P).
\]
Since the zig-zag functor is bijective on isomorphism classes, it follows that
\[
E\cong \mathcal P.
\]
\end{proof}

\begin{corollary}
\label{cor:minimal-self-dual-paper2}
The perverse sheaf
\[
\mathcal P=\Cone(\var_F)[-1]
\]
is the unique minimal Verdier self-dual perverse extension of \(\Q_U[3]\) across the ordinary
double point.
\end{corollary}

\begin{proof}
This is immediate from Theorems~\ref{thm:verdier-self-duality-paper2} and
\ref{thm:uniqueness-corrected-perverse-paper2}.
\end{proof}
Appendix \ref{app:standard-zigzags} states a list of useful zig-zags for reference.
\section{Proof of the main results}
\label{sec:proofs}

We now assemble the results established in the preceding sections and prove the main
theorems stated in the introduction.

\subsection{Proof of Theorem~\ref{thm:intro-main}}

Let
\[
\pi:\mathcal X\to\Delta
\]
be a one-parameter degeneration whose central fiber \(X_0\) has a single ordinary
double point \(p\). Let
\[
F:=\Q_{\mathcal X}[3],
\qquad
\mathcal P:=\Cone(\var_F)[-1].
\]
From \cite{RahmanSchoberPaper} the object \(\mathcal P\) is a perverse sheaf on \(X_0\), and it restricts to \(\Q_U[3]\) on the
smooth locus \(U=X_0\setminus\{p\}\) with one-dimensional point-supported singular contribution at \(p\).
By Theorems \ref{thm:verdier-self-duality-paper2} and \ref{thm:uniqueness-corrected-perverse-paper2} along with Corollary \ref{cor:minimal-self-dual-paper2} of the present paper, \(\mathcal P\) is Verdier self-dual and is the unique minimal Verdier self-dual perverse extension of \(\Q_U[3]\) across the node.

On the Hodge-theoretic side, Section~\ref{sec:LMHS} shows that nearby and vanishing cycles admit lifts to the category of mixed Hodge modules and are compatible with the realization functor
\[
\rat:MHM(X_0)\to\Perv(X_0;\Q)
\]
\cite{SaitoMHM,SaitoDuality}. Proposition~\ref{prop:mhm-nearby-cycles} shows that the limiting mixed Hodge structure and the perverse nearby-cycle construction are functorially derived from the same nearby-cycle formalism. Proposition~\ref{prop:ODP-compatible} specializes this compatibility to the ordinary double point case and shows that the rank-one vanishing contribution on the Hodge-theoretic side and the quotient
\[
i_*\Q_{\{p\}}
\]
in the short exact sequence
\[
0\to IC_{X_0}\to \mathcal P\to i_*\Q_{\{p\}}\to 0
\]
arise from the same nearby-cycle/vanishing-cycle data.

This proves Theorem~\ref{thm:intro-main}.
\qed

\subsection{Proof of Theorem~\ref{thm:intro-multinode}}

Let the central fiber \(X_0\) contain ordinary double points
\[
\Sigma=\{p_1,\dots,p_r\},
\]
and let
\[
U:=X_0\setminus\Sigma.
\]
At each node \(p_k\), the Milnor fiber has the homotopy type of \(S^3\), so the local vanishing cohomology is one-dimensional in the middle degree and vanishes otherwise \cite{MilnorSingularPoints,DimcaSheaves}. Hence the local vanishing-cycle contribution at each node is rank one.

Assume that the cone object
\[
\mathcal P=\Cone(\var_F)[-1]
\]
belongs to \(\Perv(X_0;\Q)\) and satisfies
\[
j^*\mathcal P\cong \Q_U[3].
\]
Then Proposition~\ref{prop:multinode-perverse-extension} applies and yields a short exact sequence
\[
0\longrightarrow IC_{X_0}
\longrightarrow \mathcal P
\longrightarrow \bigoplus_{k=1}^r i_{k*}\Q_{\{p_k\}}
\longrightarrow 0.
\]
This proves the structural statement claimed in Theorem~\ref{thm:intro-multinode}.

The Hodge-theoretic comparison follows from Proposition~\ref{prop:multinode-hodge}, which shows that the rank-\(r\) vanishing contribution in the limiting mixed Hodge structure and the quotient term in the above exact sequence arise from the same nearby-cycle and vanishing-cycle formalism in
Saito's category of mixed Hodge modules.

This proves Theorem~\ref{thm:intro-multinode}.
\qed

\subsection{Proof of Theorem~\ref{thm:intro-strata}}

Let the singular locus of the central fiber \(X_0\) be
\[
\Sigma=\bigsqcup_\alpha S_\alpha
\]
with a fixed Whitney stratification, and let
\[
U:=X_0\setminus\Sigma.
\]
Section~\ref{sec:strata} shows that the nearby-cycle and vanishing-cycle complexes are constructible with respect to a suitable stratification of \(X_0\), and in particular that the singular contribution is supported on \(\Sigma\) \cite{DimcaSheaves}. Assume that the corrected object
\[
\mathcal P=\Cone(\var_F)[-1]
\]
belongs to \(\Perv(X_0;\Q)\) and satisfies
\[
j^*\mathcal P\cong \Q_U[n].
\]
Then Proposition~\ref{prop:stratified-perverse-extension} yields a short exact sequence
\[
0\longrightarrow IC_{X_0}
\longrightarrow \mathcal P
\longrightarrow \mathcal V
\longrightarrow 0,
\]
where \(\mathcal V\) is a perverse sheaf supported on \(\Sigma\) and constructible with respect to the chosen stratification. This proves the structural statement of Theorem~\ref{thm:intro-strata}.

The Hodge-theoretic interpretation follows from Proposition~\ref{prop:strata-mhm}, which shows that the quotient term \(\mathcal V\) and the corresponding Hodge-theoretic singular contribution are both functorially derived from the same nearby-cycle and vanishing-cycle formalism in Saito's
theory of mixed Hodge modules.

This proves Theorem~\ref{thm:intro-strata}.
\qed

\section{Toward a K\"ahler package}
\label{sec:kahler}

For smooth projective varieties, singular cohomology carries a collection of fundamental structures
often referred to as the \emph{K\"ahler package}: Poincar\'e duality, Hodge decomposition, the Hard
Lefschetz theorem, and the Hodge--Riemann bilinear relations. These are classical consequences of
K\"ahler geometry; see, for example, \cite{GriffithsHarris,VoisinHodgeTheoryI,VoisinHodgeTheoryII}. For singular
varieties, ordinary cohomology generally fails to satisfy these properties. Intersection cohomology
was introduced by Goresky and MacPherson precisely to restore duality and Lefschetz-type
phenomena in singular settings \cite{GoreskyMacPhersonI,GoreskyMacPhersonII}. In the framework of
perverse sheaves, Beilinson, Bernstein, and Deligne proved the decomposition theorem and relative
Hard Lefschetz for projective morphisms \cite{BBD}, and subsequent work of de Cataldo and
Migliorini clarified the Hodge-theoretic content of these results and their relation to the
perverse filtration \cite{DeCataldoMigliorini,DeCataldoMiglioriniHodge}.

The corrected perverse object
\[
\mathcal P:=\Cone(\var_F)[-1]
\]
constructed in the preceding sections is intended as an analogue, in the conifold setting, of the
role played by the intersection complex in the theory of singular spaces. It is therefore natural
to ask to what extent the hypercohomology groups
\[
\mathbb H^k(X_0,\mathcal P)
\]
inherit structures analogous to the K\"ahler package. The purpose of this section is not to prove
such results, but to isolate the formal properties already available and to formulate the remaining
Hodge-theoretic problem.

\subsection{Duality}

In the ordinary double point case, the corrected perverse object \(\mathcal P\) is Verdier self-dual by
Theorem \ref{thm:verdier-self-duality-paper2}.  Verdier duality in the constructible derived category induces duality pairings on hypercohomology; see \cite{BBD,KS}. Consequently, whenever \(\mathcal P\) is Verdier self-dual, its hypercohomology groups satisfy a Poincar\'e-type duality in the derived-category sense. Thus the corrected cohomology inherits at least the formal duality structure attached to a self-dual perverse sheaf.

At the level of the present paper, this is the only part of the K\"ahler package that follows directly from the sheaf-theoretic construction of \(\mathcal P\). Stronger statements, such as Hard Lefschetz or Hodge--Riemann bilinear relations, require additional Hodge-theoretic input.

\subsection{Mixed Hodge modules and the Hodge-theoretic problem}

Saito's theory of mixed Hodge modules provides the natural framework in which one would seek such additional structure \cite{SaitoMHM,SaitoDuality}. In particular, nearby-cycle and vanishing-cycle functors are defined in the category of mixed Hodge modules and are compatible with the realization
functor
\[
\rat:MHM(X_0)\to\Perv(X_0;\Q).
\]
Thus the nearby-cycle formalism underlying the construction of \(\mathcal P\) already carries Hodge-theoretic information before one passes to the underlying perverse sheaf.

What has been established in the preceding sections is that the corrected perverse object and the relevant degeneration data on the Hodge-theoretic side arise from the same nearby-cycle and vanishing-cycle formalism. What has \emph{not} been proved in this paper is the existence of a fully internal mixed-Hodge-module object
\[
\mathcal P^H\in MHM(X_0)
\]
whose realization is \(\mathcal P\). As explained in Section~\ref{sec:ODP}, such a refinement would require an explicit gluing construction in Saito's divisor formalism. Until that step is carried out, one cannot formally deduce canonical mixed Hodge structures on \(\mathbb H^k(X_0,\mathcal P)\) merely from the existence of nearby cycles in \(MHM\).

Accordingly, the Hodge-theoretic question may be formulated as follows: does the corrected perverse object admit a mixed-Hodge-module refinement, and if so, what Hodge-theoretic structures does that refinement induce on its hypercohomology?

\subsection{Lefschetz-type expectations}

Let \(\pi:\mathcal X\to\Delta\) be a projective degeneration, and let \(L\) denote the class of a relatively ample line bundle. Relative Hard Lefschetz for perverse sheaves and mixed Hodge modules applies to the standard nearby-cycle and direct-image formalisms associated with projective morphisms \cite{BBD,SaitoMHM}. In the case of intersection cohomology, these results imply the Hard Lefschetz theorem for the intersection complex of a projective singular variety \cite{BBD,DeCataldoMigliorini}. Since the corrected object \(\mathcal P\) is constructed from the same nearby-cycle formalism, it is natural to ask whether a comparable Lefschetz theorem holds for
\[
\mathbb H^*(X_0,\mathcal P).
\]
At present, however, this should be regarded as a conjectural extension of the known formalism,
not as a theorem deduced in this paper. Establishing such a result would require either a direct
comparison with a suitable mixed-Hodge-module refinement of \(\mathcal P\), or a separate
Lefschetz theory for the corrected perverse object itself.

\subsection{Hodge--Riemann package}

The Hodge--Riemann bilinear relations for intersection cohomology are part of the Hodge theory of
algebraic maps developed by de Cataldo and Migliorini \cite{DeCataldoMiglioriniHodge}. Their
results rely on polarizable Hodge modules associated with projective morphisms and on the full
machinery of mixed Hodge modules.

In the present setting, nearby-cycle mixed Hodge modules attached to the degeneration
\(\pi:\mathcal X\to\Delta\) carry precisely the sort of Hodge-theoretic information from which one
expects a corresponding structure on the corrected cohomology to emerge. But without a fully
internal mixed-Hodge-module realization of \(\mathcal P\), the Hodge--Riemann bilinear relations
for \(\mathbb H^*(X_0,\mathcal P)\) remain conjectural.

\begin{conjecture}
Let \(\pi:\mathcal X\to\Delta\) be a projective conifold degeneration, and let \(\mathcal P\) be
the corrected perverse object constructed from nearby and vanishing cycles. If \(\mathcal P\)
admits a mixed-Hodge-module refinement compatible with the projective geometry of the degeneration,
then the hypercohomology groups
\[
\mathbb H^k(X_0,\mathcal P)
\]
should satisfy a K\"ahler-type package analogous to that of intersection cohomology, including
duality, Lefschetz-type isomorphisms, and Hodge--Riemann bilinear relations.
\end{conjecture}

\subsection{Outlook}

The preceding discussion isolates the Hodge-theoretic content of the problem. The corrected
perverse object \(\mathcal P\) already carries a formal duality structure by Verdier self-duality,
and it is constructed from nearby and vanishing cycles, which in Saito's theory possess a
natural Hodge-theoretic enhancement. The missing step is the explicit construction of a
mixed-Hodge-module refinement of \(\mathcal P\). Once such an object is constructed, one may ask
whether the resulting hypercohomology satisfies a full K\"ahler package analogous to that of
intersection cohomology.

Thus the K\"ahler-package question should be viewed as a natural continuation of the present work:
first construct the mixed-Hodge-module refinement of the corrected perverse object, and then
analyze the resulting duality, Lefschetz, and Hodge--Riemann structures on its hypercohomology.

\section{Future directions}

The results of this paper isolate several natural problems that arise from the interaction of nearby
cycles, perverse extensions, and mixed Hodge theory in degenerations. Among these, the most immediate
is the construction of a fully internal mixed-Hodge-module refinement of the corrected perverse object.

\begin{enumerate}

\item \textbf{Mixed-Hodge-module refinement of the corrected perverse object.}

The central problem left open by the present paper is the construction, in the ordinary double point
case, of an object
\[
\mathcal P^H\in MHM(X_0)
\]
whose realization under
\[
\rat:MHM(X_0)\to\Perv(X_0;\Q)
\]
is the corrected perverse object
\[
\mathcal P=\Cone(\var_F)[-1].
\]
Equivalently, one seeks an exact sequence in \(MHM(X_0)\)
\[
0\to IC^H_{X_0}\to \mathcal P^H\to i_*\Q^H_{\{p\}}(-1)\to 0
\]
refining the canonical perverse extension in the single-node case. As explained above, Saito's
divisor-case gluing formalism provides the natural setting for such a construction. Carrying out this
gluing calculation explicitly would give a fully internal Hodge-theoretic refinement of the corrected
perverse object and would provide the natural next step beyond the present paper.

\item \textbf{Multi-node gluing and global extension data.}

In the case of several ordinary double points, the singular contribution to the corrected perverse
object is a direct sum of point-supported rank-one pieces, but the global extension class need not
split as a direct sum of independent local extensions. A natural problem is therefore to describe the
global gluing data governing
\[
0\to IC_{X_0}\to \mathcal P\to \bigoplus_{k=1}^r i_{k*}\Q_{\{p_k\}}\to 0
\]
in terms of the interaction of the local vanishing cycles. One may expect this structure to reflect
the geometry of the vanishing-cycle lattice and the corresponding Picard--Lefschetz monodromy.

\item \textbf{Quiver-theoretic descriptions of multi-node degenerations.}

The multi-node case suggests an algebraic reformulation in which the local node contributions and
their extension data are organized by a quiver or diagram attached to the degeneration. Such a
description would provide a concrete framework for encoding how the point-supported vanishing terms
assemble into a single global perverse object. A quiver-theoretic model could also serve as an
intermediate step toward a more explicit analysis of the corresponding mixed-Hodge-module gluing data.

\item \textbf{Schober-theoretic and categorical refinements.}

The previous paper related the corrected perverse object to the rank-one monodromy phenomena that
also appear in the theory of spherical functors and perverse schobers. A natural next direction is
to formulate the present constructions in a schober-type language, especially in the multi-node or
stratified setting, where one expects the gluing data to admit a higher-categorical interpretation.
Such a refinement would place the corrected perverse object into a broader categorical framework
linking nearby cycles, spherical monodromy, and degeneration theory.

\item \textbf{Stratified singular loci and support decomposition.}

For singular loci with higher-dimensional strata, the corrected perverse object fits into an exact
sequence
\[
0\to IC_{X_0}\to \mathcal P\to \mathcal V\to 0,
\]
where \(\mathcal V\) is a perverse sheaf supported on the singular locus and constructible with
respect to the chosen stratification. A natural problem is to determine whether, under additional
hypotheses, \(\mathcal V\) admits a more explicit decomposition by strict support or by local systems
on the strata. This would clarify how the local Milnor-fiber data along the strata is assembled into
the global singular contribution.

\item \textbf{Lefschetz and Hodge--Riemann structures on corrected cohomology.}

A longer-term goal is to determine whether the hypercohomology groups
\[
\mathbb H^k(X_0,\mathcal P)
\]
satisfy a K\"ahler-type package analogous to that of intersection cohomology. The present paper
isolates the formal duality inherited from Verdier self-duality and identifies the nearby-cycle
formalism that underlies the corrected object. The next step would be to combine a mixed-Hodge-module
refinement of \(\mathcal P\) with projective Hodge theory in order to investigate Hard Lefschetz and
Hodge--Riemann bilinear relations for the corrected cohomology.

\item \textbf{Relations with stringy and singularity-corrected Hodge theories.}

The conifold setting connects the present constructions with other proposals for corrected cohomology
theories associated with singular spaces, including intersection-space cohomology and related Hodge-theoretic
refinements \cite{BanaglBudurMaxim,BanaglIS}. It would be valuable to clarify more precisely how the
corrected perverse object studied here compares with these constructions, especially at the level of
mixed Hodge structures and degeneration data.

\end{enumerate}

Taken together, these problems suggest that the corrected perverse object is only the first layer of a
broader structure linking nearby cycles, degeneration theory, mixed Hodge modules, and categorical
monodromy. The most immediate next step is the explicit mixed-Hodge-module refinement in the
single-node case; the remaining directions may be viewed as successive extensions of that program.
%
%
\appendix

\section{MacPherson--Vilonen zig-zags in the ordinary double point case}
\label{app:standard-zigzags}

This appendix records the standard zig-zag models used in the ordinary double point case using formalism using formalism from $\S$4 in \cite{RahmanATMP}.
They are collected here both for convenience and for later comparison with mixed-Hodge-module
shadows, finite-node direct sums, and future categorical/schober shadow data.

\subsection{Convention}

We use the MacPherson--Vilonen zig-zag functor
\[
\mu:\Perv(X_0;\Q)\longrightarrow Z(X_0,p)
\]
in the isolated-stratum convention of \cite{RahmanATMP}. For
\[
K\in\Perv(X_0;\Q),
\]
we write
\[
\mu(K)=(L_K,A_K,B_K,\alpha_K,\beta_K,\gamma_K),
\]
where \(L_K=j^*K\) is the open-stratum part and
\[
H^{-1}\!\bigl(i^*Rj_*L_K\bigr)\xrightarrow{\alpha_K}
A_K\xrightarrow{\beta_K}B_K\xrightarrow{\gamma_K}
H^0\!\bigl(i^*Rj_*L_K\bigr)
\]
is the associated exact zig-zag sequence.

\subsection{The minimal extension object}

\begin{proposition}
\label{prop:appendix-IC}
The intersection-complex perverse sheaf has zig-zag
\[
\mu(IC_{X_0})\cong (\Q_U[3],0,0,0,0,0).
\]
\end{proposition}

\begin{proof}
Since \(IC_{X_0}=j_{!*}\Q_U[3]\) is the minimal extension of the constant perverse sheaf on the
smooth stratum, the point terms vanish in the MacPherson--Vilonen model. Thus
\[
\mu(IC_{X_0})\cong (\Q_U[3],0,0,0,0,0).
\]
\end{proof}

\subsection{The point-supported rank-one object}

\begin{proposition}
\label{prop:appendix-skyscraper}
The point-supported perverse sheaf \(i_*\Q_{\{p\}}\) has zig-zag
\[
\mu(i_*\Q_{\{p\}})\cong (0,\Q,\Q,0,\id,0).
\]
\end{proposition}

\begin{proof}
Because \(i_*\Q_{\{p\}}\) is supported entirely at \(p\), the open-stratum part vanishes. The
MacPherson--Vilonen exact sequence therefore collapses to the point-supported rank-one case,
yielding
\[
\mu(i_*\Q_{\{p\}})\cong (0,\Q,\Q,0,\id,0).
\]
\end{proof}

\subsection{Split and non-split extensions}

\begin{proposition}
\label{prop:appendix-split}
The split extension
\[
IC_{X_0}\oplus i_*\Q_{\{p\}}
\]
has open-stratum part \(\Q_U[3]\), one-dimensional point terms, and trivial extension class.
\end{proposition}

\begin{proof}
The object \(IC_{X_0}\oplus i_*\Q_{\{p\}}\) is obtained by taking the direct sum of the endpoint
objects. Its zig-zag therefore has the same open-stratum term as \(IC_{X_0}\), together with the
rank-one point terms contributed by \(i_*\Q_{\{p\}}\), but with trivial gluing class.
\end{proof}

\begin{proposition}
\label{prop:appendix-generic-extension}
Let
\[
0\longrightarrow IC_{X_0}\longrightarrow E\longrightarrow i_*\Q_{\{p\}}\longrightarrow 0
\]
be a general extension. Then \(\mu(E)\) is an extension zig-zag of
\[
(0,\Q,\Q,0,\id,0)
\]
by
\[
(\Q_U[3],0,0,0,0,0),
\]
with open-stratum term \(\Q_U[3]\). After choosing splittings, the point terms are one-dimensional
and the extension class is encoded in the corresponding gluing parameter.
\end{proposition}

\begin{proof}
Applying \(\mu\) to the short exact sequence gives an extension zig-zag of the indicated endpoint
objects. The ordinary double point case is rank one at the singular point, so the point terms are
one-dimensional. The distinction between split and non-split extensions is therefore carried by the
gluing class rather than by the ambient dimensions.
\end{proof}

\subsection{The corrected perverse object}

\begin{proposition}
\label{prop:appendix-corrected}
The corrected perverse object
\[
\mathcal P:=\Cone(\var_F)[-1]
\]
has zig-zag
\[
\mu(\mathcal P)\cong (\Q_U[3],\Q,\Q,0,\id,0).
\]
\end{proposition}

\begin{proof}
This is Proposition~\ref{prop:full-zigzag-P}. The corrected object is the unique non-split
extension of \(i_*\Q_{\{p\}}\) by \(IC_{X_0}\), so its zig-zag is the unique nontrivial
extension zig-zag of the corresponding endpoint objects.
\end{proof}

\subsection{Duality}

\begin{proposition}
\label{prop:appendix-duality}
The zig-zags of \(IC_{X_0}\), \(i_*\Q_{\{p\}}\), and \(\mathcal P\) are self-dual under the zig-zag
duality functor of \cite{RahmanATMP}.
\end{proposition}

\begin{proof}
For \(IC_{X_0}\) and \(i_*\Q_{\{p\}}\), this follows from the explicit zig-zag calculations above.
For \(\mathcal P\), this follows from Theorem~\ref{thm:verdier-self-duality-paper2}.
\end{proof}

\subsection{Multi-node direct sums}

\begin{proposition}
\label{prop:appendix-multinode}
For a finite set of isolated nodes \(\Sigma=\{p_1,\dots,p_r\}\), the direct-sum point-supported object
\[
\bigoplus_{k=1}^r i_{k*}\Q_{\{p_k\}}
\]
has shadow given by the direct sum of the corresponding rank-one point zig-zags.
\end{proposition}

\begin{proof}
Each \(i_{k*}\Q_{\{p_k\}}\) contributes the rank-one point zig-zag
\[
(0,\Q,\Q,0,\id,0),
\]
and the direct sum is obtained componentwise.
\end{proof}

\subsection{Table of standard zig-zags}

\begin{center}
\footnotesize
\renewcommand{\arraystretch}{1.2}
\setlength{\tabcolsep}{4pt}
\captionof{table}{Standard zig-zags in the ordinary double point case.}
\begin{tabular}{|C{0.23\textwidth}|C{0.39\textwidth}|C{0.28\textwidth}|}
\hline
\textbf{Object} & \textbf{Zig-zag} & \textbf{Comments} \\
\hline
\(IC_{X_0}\) &
\((\Q_U[3],0,0,0,0,0)\) &
minimal extension \\
\hline
\(i_*\Q_{\{p\}}\) &
\((0,\Q,\Q,0,\id,0)\) &
point-supported rank-one object \\
\hline
\(\mathcal P:=\Cone(\var_F)[-1]\) &
\((\Q_U[3],\Q,\Q,0,\id,0)\) &
unique corrected non-split class \\
\hline
\(\bigoplus_{k=1}^r i_{k*}\Q_{\{p_k\}}\) &
\(\bigoplus_{k=1}^r (0,\Q,\Q,0,\id,0)\) &
multi-node local shadow \\
\hline
\end{tabular}
\end{center}

\medskip

We emphasize that in the compressed rank-one notation, the split and non-split extensions may have
the same ambient zig-zag shape; they are distinguished by their extension class, equivalently by the
associated gluing parameter in the matrix-form description.

\medskip
\begin{proposition}
\label{prop:compressed-zigzag-not-complete}
Let
\[
0\to K'\to K\to K''\to 0
\]
be a short exact sequence in \(\Perv(X_0;\Q)\), and let
\[
\mu(K')=(L',A',B',\alpha',\beta',\gamma'),
\qquad
\mu(K'')=(L'',A'',B'',\alpha'',\beta'',\gamma'')
\]
be the corresponding MacPherson--Vilonen zig-zags. Then the zig-zag of \(K\) is determined by an
extension of the endpoint zig-zags together with the induced gluing data in the corresponding maps.
In particular, the compressed ambient zig-zag shape of \(K\), that is, the tuple obtained by recording
only the open-stratum term and the ambient point terms, does not in general determine the isomorphism class of \(K\). Equivalently, two non-isomorphic extension objects may have the same compressed ambient zig-zag shape while differing in their extension class.
\end{proposition}
\begin{proof}
Applying the MacPherson--Vilonen zig-zag functor \(\mu\) to a short exact sequence
\[
0\to K'\to K\to K''\to 0
\]
produces an extension of the corresponding zig-zag data. In particular, the open-stratum term and
the point terms fit into exact sequences
\[
0\to L'\to L\to L''\to 0,
\qquad
0\to A'\to A\to A''\to 0,
\qquad
0\to B'\to B\to B''\to 0,
\]
where
\[
\mu(K)=(L,A,B,\alpha,\beta,\gamma).
\]

After choosing splittings of the underlying vector-space extensions, one may identify
\[
A\cong A'\oplus A'',
\qquad
B\cong B'\oplus B'',
\]
so that the induced map \(\beta\) is represented by a block matrix whose off-diagonal term records
the extension class. Thus the ambient dimensions of \(L\), \(A\), and \(B\), and even the resulting
compressed zig-zag shape, do not by themselves determine the isomorphism class of \(K\). The
missing information is precisely the gluing data carried by the induced maps, equivalently the
extension class. Therefore two extension objects may have the same compressed ambient zig-zag shape while still being non-isomorphic.
\end{proof}
Table~2 illustrates Proposition~\ref{prop:compressed-zigzag-not-complete} in the ordinary double
point case, where the split extension and the corrected non-split extension have the same compressed
ambient zig-zag shape but are distinguished by their extension class.
\begin{center}
\scriptsize
\renewcommand{\arraystretch}{1.2}
\setlength{\tabcolsep}{3pt}
\captionof{table}{Extension templates in zig-zag form.}
\begin{tabular}{|C{0.33\textwidth}|C{0.33\textwidth}|C{0.24\textwidth}|}
\hline
\textbf{Object} & \textbf{Zig-zag} & \textbf{Comments} \\
\hline
\parbox[c]{\linewidth}{\centering split extension\\[-2pt]
{\tiny \(0\to IC_{X_0}\to IC_{X_0}\oplus i_*\Q_{\{p\}}\to i_*\Q_{\{p\}}\to 0\)}} &
\((\Q_U[3],\Q,\Q,0,\id,0)\) &
trivial extension class \\
\hline
\parbox[c]{\linewidth}{\centering general extension \(E\)\\[-2pt]
{\tiny \(0\to IC_{X_0}\to E\to i_*\Q_{\{p\}}\to 0\)}} &
\(\left(\Q_U[3],\,A\oplus\Q,\,B\oplus\Q,\,\alpha_E,\,
\begin{psmallmatrix}\beta & u\\[2pt] 0 & 1\end{psmallmatrix},\,\gamma_E\right)\) &
Here \(A\) and \(B\) denote the point terms in a general endpoint zig-zag
\((\Q_U[3],A,B,\alpha,\beta,\gamma)\). For the specific object \(IC_{X_0}\) considered in this appendix,
one has \(A=B=0\), so the ordinary double point specialization collapses to the rank-one form listed
for \(\mathcal P\). The parameter \(u\in B\) records the extension class modulo \(\operatorname{Im}\beta\). \\
\hline
\parbox[c]{\linewidth}{\centering corrected non-split extension \(\mathcal P\)\\[-2pt]
{\tiny \(0\to IC_{X_0}\to \mathcal P\to i_*\Q_{\{p\}}\to 0\)}} &
\((\Q_U[3],\Q,\Q,0,\id,0)\) &
unique nontrivial self-dual class \\
\hline
\end{tabular}
\end{center}
%
%
\printbibliography
\end{document}